\documentclass[letterpaper, 10 pt, conference]{cls/ieeeconf}
\IEEEoverridecommandlockouts
\usepackage[table,usenames,dvipsnames]{xcolor}      
\usepackage[noadjust]{cite}
\usepackage{amsmath,amssymb,amsfonts,amsthm,dsfont,mathtools}
\usepackage{nicematrix}

\usepackage{adjustbox}
\usepackage{graphicx,tabularx,adjustbox,color}
\usepackage{multirow}
\usepackage[font=footnotesize]{caption}
\usepackage[font=footnotesize]{subcaption}

\usepackage{algorithmic}
\usepackage[ruled,vlined]{algorithm2e}
\usepackage{stackengine}

\usepackage{dirtytalk}
\allowdisplaybreaks
\renewcommand{\baselinestretch}{.985}

\usepackage[breaklinks=true, colorlinks, bookmarks=true, citecolor=Black, urlcolor=Violet,linkcolor=Black]{hyperref}

\DeclareMathOperator*{\diag}{diag}

\newtheorem{proposition}{Proposition}
\newtheorem{lemma}{Lemma}
\newtheorem{theorem}{Theorem}
\theoremstyle{definition}

\newtheorem*{problem*}{Problem}

\newtheorem*{remark*}{Remark}
\newtheorem{assum}{Assumption}

\newcommand{\ba}{\begin{align}}
\newcommand{\ea}{\end{align}}

\newcommand{\ep}{\varepsilon}

\newcommand{\bbR}{\mathbb{R}}

\title{\LARGE \bf 
Safe Stabilization of the Stefan Problem with a High-Order Moving Boundary Dynamics by PDE Backstepping
} 

\author{
Shumon Koga, Miroslav Krstic
\thanks{S. Koga is with the Department
of Computer Science and Systems Engineering at Kobe University,
Hyogo, Japan (e-mail: koga@harbor.kobe-u.ac.jp).} 
\thanks{M. Krstic is with the Department
of Mechanical and Aerospace Engineering, University of California at San Diego, La Jolla,
CA, 92093-0411 USA (e-mail: krstic@ucsd.edu).}
\thanks{This work was partially supported by JSPS KAKENHI Grant Number JP25058336.}}

\begin{document}
\maketitle
\begin{abstract}
This paper presents a safe stabilization of the Stefan PDE model with a moving boundary governed by a high-order dynamics. We consider a parabolic PDE with a time-varying domain governed by a second-order response with respect to the Neumann boundary value of the PDE state at the moving boundary. The objective is to design a boundary heat flux control to stabilize the moving boundary at a desired setpoint, with satisfying the required conditions of the model on PDE state and the moving boundary. We apply a PDE backstepping method for the control design with considering a constraint on the control law. The PDE and moving boundary constraints are shown to be satisfied by applying the maximum principle for parabolic PDEs. Then the closed-loop system is shown to be globally exponentially stable by performing Lyapunov analysis. The proposed control is implemented in numerical simulation, which illustrates the desired performance in safety and stability. An outline of the extension to third-order moving boundary dynamics is also presented. Code is released at \url{https://github.com/shumon0423/HighOrderStefan_CDC2025.git}. 
\end{abstract}

\section{Introduction}
\label{sec: intro}

The Stefan problem is a well-studied model representing the dynamics of thermal phase change, which has been widely utilized in applications including sea-ice melting \cite{koga2019arctic}, cryosurgery \cite{Rabin1998},  cell thawing \cite{srisuma2023thermal}, additive manufacturing \cite{wang2021closed}, electrosurgery probe-tissue \cite{el2021pde}. The Stefan problem comprises a parabolic Partial Differential Equation (PDE) of the temperature state defined on a moving boundary of the liquid-solid interface governed by an Ordinary Differential Equation (ODE). Such Stefan-type moving boundary models have been proposed in chemical, biological, and sociological fields such as lithium-ion batteries \cite{pozzato2024accelerating},  neuron growth \cite{demir2024neuron}, and epidemic model \cite{zhuang2021spatial}. 

While the Stefan problem has been studied for more than a century, due to the challenge of handling the nonlinearity arising in the moving boundary dynamics, there have been ongoing work studying the solution of the Stefan problem by the state-of-art methods such as Bayesian optimization \cite{winter2023multi} and physics-informed machine learning \cite{wang2021deep}. Designing a boundary feedback control of the Stefan problem has also been proposed with several approaches such as on-off switching \cite{Hoffman82}, optimal control \cite{Hinze07}, motion planning by series expansion \cite{dunbar2003boundary}, geometric control \cite{maidi2014}, enthalpy-based control \cite{petrus12}, trajectory tracking \cite{ecklebe2021toward}. Recently, a backstepping-based boundary control has been developed for the Stefan problem in \cite{Shumon19journal}, which achieves the global stabilization of closed-loop system. Owing to the systematic property of the PDE backstepping method \cite{PDEbook}, several extensions of the model and design have been developed, as summarized in \cite{KKbook2021}. 

A key approach done in \cite{Shumon19journal} is guaranteeing conditions to maintain the validity of the physical model of the liquid and solid phases, by applying the maximum principle of parabolic PDEs \cite{pao2012nonlinear}. Such a condition to ensure for the system's requirement has been recently categorized as a notion of safety together with the Control Barrier Function (CBF) \cite{ames2016control}. For systems with a high relative degree, several approaches have been proposed to address the safety such as exponential CBF \cite{nguyen2016exponential}, high-order CBF \cite{xiao2021high}, and nonovershooting \cite{krstic2006nonovershooting}. 

Applying CBF-based methods to PDE systems has been initially proposed for the Stefan problem \cite{koga2023safe}, by adding actuator dynamics and developing a nonovershooting control and exponential CBF-Quadratic Programming (QP) safety filter, where the maximum principle ensures PDE safety condition once the condition on the boundary value is satisfied by CBF on actuator dynamics. Digital implementation by event-triggered mechanism of \cite{koga2023safe} has been achieved in \cite{koga2023event}. Developing an adaptive safe control for hyperbolic PDE with actuator dynamics has been done in \cite{wang2024safe}. However, none has studied the CBF of the moving boundary of a high-relative degree. 

Without safety requirement, a local stabilization of PDE with a moving boundary has been tacked in literature. \cite{buisson2018control} has proposed a stabilization of the piston position represented as a moving boundary governed by a second-order dynamics of two first-order hyperbolic PDE. In \cite{yu2020bilateral}, the authors have developed a boundary control of traffic congestion model with a moving shock shockwave described as a moving boundary governed by a first-order dynamics together with a first order hyperbolic PDE. For parabolic PDEs, \cite{demir2024neuron} have tackled the stabilization of the neuron growth model given by a parabolic PDE with a moving boundary governed by a second-order dynamics. 

This paper is the first studying the safe stabilization of a high-order moving boundary dynamics governed together with the Stefan PDE. The contributions are (i) formulating a second-order moving boundary dynamics and showing a condition to ensure the constraint on moving boundary, (ii) developing a control law to stabilize the moving boundary with proofs of satisfying the given constraints and the global exponential stability by PDE backstepping and Lyapunov analysis, (iii) providing an outline of the extension to higher-order moving boundary dynamics. 

Section \ref{sec:problem} presents the Stefan PDE model with a second-order moving boundary dynamics and show a lemma to guarantee the constraint on the moving boundary. Section \ref{sec:control} provides the derivation of the control law to stabilize the moving boundary with a proof of control constraint. Section \ref{sec:stability} gives our main theorem with showing the proof of stability. Section \ref{sec:simulation} provides the simulation result of the proposed control design. Section \ref{sec:highorder} remarks an outline of the extension to higher-order moving boundary dynamics. Section \ref{sec:conclusion} gives the conclusion and future work. 

\newcommand{\intst}{\int_0^{s(t)}}

\section{Problem} \label{sec:problem}

\begin{figure}[t]
\centering
\includegraphics[width=0.99\linewidth]{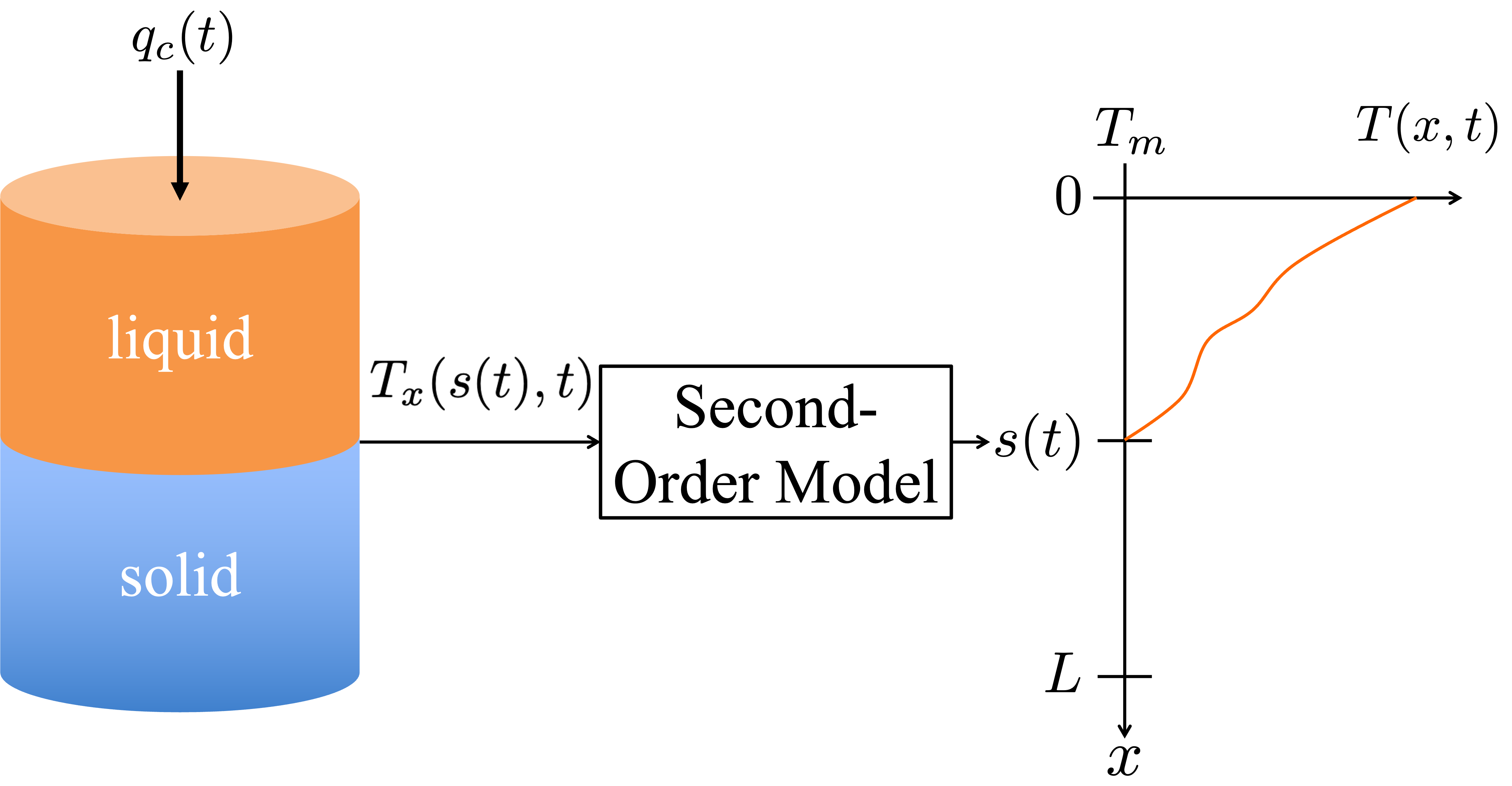}\\
\caption{Schematic of Stefan problem with a second-order moving boundary dynamics.}
\label{fig:stefan}
\end{figure}

Consider the melting in a material of length $L$ in one dimension which contains the liquid phase on the domain $[0, s(t)]$ and the solid phase on the domain $[s(t), L]$ (see  Fig.~\ref{fig:stefan}). 
The energy conservation and heat conduction laws yield 
the heat equation of  the temperature $T(x,t)$ in the liquid phase, the boundary conditions, the dynamics of the moving boundary, and the initial values as follows:

\begin{align}\label{eq:stefanPDE}
T_t(x,t)&=\alpha T_{xx}(x,t), \hspace{2mm}  \textrm{for} \hspace{2mm} t > 0, \hspace{2mm} 0< x< s(t), \\ 
\label{eq:stefancontrol}
-k T_x(0,t)&=q_{\rm c}(t),  \hspace{2mm} \textrm{for} \hspace{2mm} t >0,\\ \label{eq:stefanBC}
T(s(t),t)&= T_{\rm m}, \hspace{2mm} \textrm{for} \hspace{2mm} t >0, \\
\label{eq:stefanIC}
 s(0) &=  s_0, \textrm{and } T(x,0) = T_0(x),  \textrm{for } x \in (0, s_0]. 
\end{align}

The classical one-phase Stefan problem yields the dynamics of the moving boundary as a first-order response (actually an integrator response) with respect to the heat flux at the moving boundary.
In this paper, we consider the second-order response represented by 
\begin{align}
\label{eq:stefanODE}
\ep \ddot s(t) & = - \dot s(t) - \beta T_x(s(t),t) , \\
\dot s(0) & = v_0, 
\end{align}
where $\ep$ denotes the relaxation time 
in phase transformation and $\ep \ddot s(t)$ represents the inertia of thermal process. 

There are two requirements for the validity of the model \eqref{eq:stefanPDE}-\eqref{eq:stefanODE}:
\begin{align}\label{temp-valid}
T(x,t) \geq& T_{{\rm m}}, \quad  \forall x\in(0,s(t)), \quad \forall t>0, \\
\label{int-valid}0 < s(t)<  &L, \quad \forall t>0. 
\end{align}
First, the trivial: the liquid phase is not frozen, i.e., the liquid temperature $T(x,t)$ is greater than the melting temperature $T_{\rm m}$. Second, equally trivially, the material is not entirely in one phase, i.e., the interface remains inside the material's domain. These physical conditions are also required for the existence and uniqueness of solutions.
Hence, we assume the following for the initial data. 

\begin{assum}\label{ass:initial} 
$0 < s_0 < L$, $T_0(x) \in C^1([0, s_0];[T_{\rm m}, +\infty))$ with $T_0(s_0) = T_{\rm m}$.
 \end{assum}
 \begin{assum}\label{ass:CBFinitial}
     $v_0 \geq 0$. 
 \end{assum}

 Assumption \ref{ass:initial} is a classical one noted in most of the literature about the Stefan problem, while Assumption \ref{ass:CBFinitial} is a condition we introduce in this paper to handle the second-order moving boundary dynamics. 

We note the following lemma to ensure the well-posedness of the solution satisfying the requirements. 
 \begin{lemma}\label{lem1}
With Assumptions \ref{ass:initial}-\ref{ass:CBFinitial}, if $q_{\rm c}(t)$ is a bounded piecewise continuous non-negative heat function, i.e.,
\begin{align} \label{eq:qct-valid}
q_{{\rm c}}(t) \geq 0,  \quad \forall t\geq 0,  
\end{align}
then there exists a unique classical solution for the Stefan problem \eqref{eq:stefanPDE}--\eqref{eq:stefanODE}, which satisfies \eqref{temp-valid}, and 
\begin{align} \label{eq:sdot-pos} 
    \dot s(t) \geq 0, \quad \forall t \geq 0. 
\end{align}

\end{lemma}

\begin{proof}
    The definition of the classical solution of the Stefan problem is given in 
Appendix A of \cite{Shumon19journal}. We apply a similar approach to the proof of the well-posedness of the classical Stefan problem with first-order moving boundary dynamics. By applying the maximum principle for parabolic PDEs and Hopf's lemma, it holds that $T_{x}(s(t),t) \leq 0$, which leads to 
\begin{align}
    \ep \ddot s(t) \geq - \dot s(t). 
\end{align}
Applying the Gronwall's inequality yields $
    \dot s(t) \geq v_0 e^{ - \ep^{-1} t}$. 
Thus, with Assumption \ref{ass:CBFinitial}, \eqref{eq:sdot-pos} holds. 

\end{proof}

The objective of the paper is to develop a boundary heat flux control law $q_{\rm c}(t)$ to achieve the regulation of the moving boundary $s(t)$ at a desired setpoint $s_{\rm r}$. The derivation of the control law is given in the next section.

\section{Control Derivation} \label{sec:control}

This section provides the derivation of the control law associated with the PDE backstepping method. 
\subsection{Backstepping Transformation}
Let $u(x,t) \in \bbR_+$ and $X(t) \in \bbR^2$ be the reference error states defined by
\begin{align} \label{eq:uX-def}
    u(x,t) &= T(x,t) - T_{\rm m}, \quad 
    X = 
    \left[
    \begin{array}{cc}
    s(t) - s_{\rm r} \\
     \dot s(t) 
    \end{array}
    \right]
\end{align}
Rewriting the system \eqref{eq:stefanPDE}--\eqref{eq:stefanODE} with respect to the reference error states \eqref{eq:uX-def} yields the reference error system as 
\begin{align} \label{eq:u-PDE}
    u_t &= \alpha u_{xx} , \\
\label{eq:u-BC1}  u_x(0,t) & = - k^{-1} q_{\rm c}(t) , \\
  \label{eq:u-BC2}  u(s(t),t) & = 0, \\
 \label{eq:u-ODE} \dot X(t) & = AX(t) + B u_x(s(t),t),  
\end{align}
where 
\begin{align}
    A = \left [
    \begin{array}{cc}
       0  & 1 \\
       0  & - \ep^{-1}
    \end{array}
    \right] , \quad 
    B =\left [
    \begin{array}{c}
       0  \\
       - \ep^{-1} \beta
    \end{array}
    \right]  . 
\end{align}
We introduce the backstepping transformation given by 
\begin{align}
    w(x,t) & = u(x,t) -  \int_x^{s(t)} k(x - y) u(y,t) dy \notag\\
    & \quad - \phi^\top (x-s(t)) X(t) \label{eq:bkst}
\end{align}
where $k(x)$ and $\phi(x)$ are the gain kernel functions to be determined. Throughout the transformation \eqref{eq:bkst}, the target system of $(w,X)$ states is chosen to govern the following dynamics 
\begin{align} 
    w_t &= \alpha w_{xx} + C^\top(x-s(t)) X(t) \notag\\
    & \quad + \dot s(t) \phi'(x-s(t)) X(t), \label{eq:w-PDE} \\
  w_x(0,t) & = 0 , \label{eq:w-BC1} \\
    w(s(t),t) & = D^\top X(t), \label{eq:w-BC2}\\
    \dot X(t) & = (A + BK) X(t) + B w_x(s(t),t). \label{eq:w-ODE}
\end{align} 
where $C(x) \in \bbR^2$ and $D \in \bbR^2$ are to be determined, and $K \in \bbR^2$ is chosen to make $A + BK$ a Hurwiz matrix. 
Taking the spatial and time derivatives of \eqref{eq:bkst} along with the system \eqref{eq:u-PDE}--\eqref{eq:u-ODE}, to satisfy \eqref{eq:w-PDE}, we can derive the following conditions for the gain kernel functions: 
\begin{align} \label{eq:k-cond}
    k(x) &= - \frac{1}{\alpha} \phi(x)^\top B, \\
    \alpha \phi^{''\top}(x) &=  \phi^\top(x) A + C(x)^\top . \label{eq:phi-ODE}
\end{align}
Substituting $x = s(t)$ into the transformation \eqref{eq:bkst} and its spatial derivative with plugging the boundary conditions \eqref{eq:u-BC2} and \eqref{eq:u-ODE}, to satisfy \eqref{eq:w-BC2} and \eqref{eq:w-ODE}, we derive the following conditions: 
\begin{align} \label{eq:phi-cond}
    \phi(0) = - D, \quad \phi'(0) = K.   
\end{align}
The conditions \eqref{eq:k-cond}--\eqref{eq:phi-cond} provide a unique closed-form solution of the kernel functions $k(x)$ and $\phi(x)$. Substituting $x = 0$ into the spatial derivative of \eqref{eq:bkst} with using \eqref{eq:u-BC1}, \eqref{eq:w-BC1}, and \eqref{eq:k-cond} leads to the boundary control law: 
\begin{align}
    q_{\rm c}(t) & = \frac{k}{\alpha} D^\top B  u(0) + \frac{k}{\alpha}  \int_0^{s(t)} \phi'(- y)^\top B u(y,t) dy \notag\\
    & \quad - k \phi^{'\top} (-s(t)) X(t).  \label{eq:qct}
    \end{align}

\subsection{Positivity Condition of Boundary Control}
While the boundary control law \eqref{eq:qct} gives the general formulation for any given pair of vectors $(C,D)$, it is necessary to care about the validity conditions \eqref{temp-valid} and \eqref{int-valid} of the Stefan problem in the closed-loop system, which is ensured by positivity of the boundary control $q_{\rm c}(t) \geq 0$. For this reason, we choose the vectors $(C,D)$ as 
\begin{align} \label{eq:C-def}
    C(x) = - A^\top \phi(x), \quad D = 0. 
\end{align}
Then, the solution to \eqref{eq:phi-ODE}--\eqref{eq:phi-cond} is given by  
\begin{align} \label{eq:phi-sol}
    \phi(x) = K x. 
\end{align}
Let $K \in \bbR^2$ be 
\begin{align}
    K^\top = \frac{1}{\beta}
    \left[
    \begin{array}{cc}
      c_1   &  \ep c_2 
    \end{array}
    \right],  \label{eq:K-form}
\end{align}
where $(c_1, c_2)$ are gain parameters. Then, substituting the solution \eqref{eq:phi-sol} with \eqref{eq:K-form} into the boundary control law \eqref{eq:qct} yields the explicit form of the control law as 
\begin{align}
    q_{\rm c}(t)
    & = - \frac{ k c_2}{\alpha}  \int_0^{s(t)}  (T(x,t) - T_{\rm m})  dx \notag\\
    & \quad - \frac{k}{\beta} \left (c_1 (s(t) - s_{\rm r}) + c_2\ep  \dot s(t) \right).  \label{eq:qct-exp}
\end{align}
We impose the following assumption for the choice of the setpoint. 
\begin{assum} \label{ass:setpoint}
The setpoint $s_{\rm r}$ is chosen to satisfy
    \begin{align}
     s_{\rm r} > \underline{s}_{\rm r}(s_0, v_0, T_0)  , 
    \end{align}
    where 
    \begin{eqnarray}
    \underline{s}_{\rm r}(s_0, v_0, T_0) &:=& s_0 + \ep v_0 + \frac{  \beta}{ \alpha }  \int_0^{s_0}  (T_0(x) - T_{\rm m})  dx
    \nonumber\\
    &\geq & s_0 \,. \label{eq:setpoint-bound}    
    \end{eqnarray} 
\end{assum}

\begin{assum} \label{ass:gain-safe}
    The control gains are chosen to satisfy 
    \begin{align} \label{eq:gain-cond}
        c_1 \leq c_2 < \underbrace{c_1 \left(1+\frac{s_{\rm r} - \underline{s}_{\rm r}(s_0, v_0, T_0)}{ \underline{s}_{\rm r}(s_0, v_0, T_0) - s_0}\right)}_{\in(c_1,+\infty]}\,.
    \end{align}
\end{assum}

We show the following lemma to guarantee the conditions: 
\begin{proposition} \label{prop:safety}
    Let Assumptions \ref{ass:initial}--\ref{ass:gain-safe} hold. Then, the closed-loop system of \eqref{eq:stefanPDE}--\eqref{eq:stefanODE} with the control law \eqref{eq:qct-exp} 
    has a unique solution satisfying \eqref{temp-valid}, \eqref{eq:qct-valid} and 
    \begin{align} \label{eq:st-cond}
        s_0 \leq s(t) \leq s_{\rm r}, \quad \forall t \geq 0. 
    \end{align}
\end{proposition}
\begin{proof}
    Taking the time derivative of the control law \eqref{eq:qct-exp} and plugging \eqref{eq:u-PDE}--\eqref{eq:u-ODE} leads to  
\begin{align}
    \dot q_{\rm c}(t)
    & =  - c_2 q_{\rm c}(t) + \frac{k}{\beta} (c_2 - c_1 ) \dot s(t), \label{eq:qcdot}
\end{align}
We show the positivity condition \eqref{eq:qct-valid} by contradiction. Assume that there exists $t^* >0$ such that $q_{\rm c}(t) > 0$ for all $t \in [0, t^*)$ and $q_{\rm c}(t^*) = 0$. Then, it holds that $\dot s(t) >0 $ for all $t \in [0, t^*)$. Thus, with the condition \eqref{eq:gain-cond}, 
it holds that 
\begin{align}
   \dot q_{\rm c}(t) > - c_2 q_{\rm c}(t), \quad \forall t \in [0, t^*). 
\end{align}
Applying Gronwall's inequality, we have 
\begin{align}
    q_{\rm c}(t) > q_{\rm c}(0) e^{-ct}, \quad \forall t \in [0, t^*). 
\end{align}
Therefore, it also holds that $q_{\rm c}(t^*) \geq q_{\rm c}(0) e^{- c t^*} > 0$, which contradicts with $q_{\rm c}(t^*)= 0$. Hence, \eqref{eq:qct-valid} holds for all $t \geq 0$. Moreover, following Lemma \ref{lem1}, the validity conditions \eqref{temp-valid} and \eqref{eq:sdot-pos} holds for all $t \geq 0$. Applying those conditions to \eqref{eq:qct-exp} leads to \eqref{eq:st-cond}. 
\end{proof}

\section{Lyapunov Analysis} \label{sec:stability}

This section provides stability proof of the closed-loop system under the control law utilizing the safety conditions shown in previous sections. We state the following theorem. 
\begin{theorem}\label{thm:stability}
    Let Assumptions \ref{ass:initial}--\ref{ass:setpoint} hold. Consider the closed-loop system of \eqref{eq:stefanPDE}--\eqref{eq:stefanODE} with the control law \eqref{eq:qct-exp}.  Then, with the control gains satisfying 
    \begin{align}
    0<  c_1 & \leq   c_2 < 
        \begin{cases}
           c_1 + \overline c_2 (s_0, v_0, T_0, s_{\rm r}), \quad {\rm if} \quad 12 s_{\rm r}^2\leq \alpha \ep, \\
           c_1 +\overline{\overline{c}}_2(s_0, v_0, T_0, s_{\rm r}) , \quad {\rm otherwise}
        \end{cases} \label{eq:theorem-cond}
    \end{align}
    where 
     \begin{align}
          \overline c_2 (s_0, v_0, T_0, s_{\rm r}) & : = c_1\frac{s_{\rm r} - \underline{s}_{\rm r}(s_0, v_0, T_0)}{ \underline{s}_{\rm r}(s_0, v_0, T_0) - s_0}
          \\
          \overline{\overline{c}}_2(s_0, v_0, T_0, s_{\rm r}) &:= \min\left\{ \overline c_2 (s_0, v_0, T_0, s_{\rm r}), \frac{\alpha \ep c_1 +  \alpha }{12 s_{\rm r}^2 - \alpha \ep} \right\}\,,
     \end{align}
    the closed-loop system is exponentially stable in $H_1$-norm, i.e., there exist positive constants $M>0, b>0$ such that  $
    \Phi(t) \leq M \Phi(0) e^{-bt}$ 
holds for the Lyapunov function $\Phi(t) := || w||^2 + || w_x||^2 + X^\top X$. 
\end{theorem}

The proof of Theorem 1 is shown in the remainder of this section. 
With \eqref{eq:C-def}, the target system is rewritten as  
\begin{align} \label{eq:target-special-PDE}
    w_t &= \alpha w_{xx} - (x - s(t)) K^\top A X(t) \notag\\
    & \quad + \dot s(t) K^\top X(t), \\
 \label{eq:target-special-BC1} w_x(0,t) & = 0 , \\
 \label{eq:target-special-BC2}   w(s(t),t) & = 0, \\
 \label{eq:target-special-ODE}   \dot X(t) & = (A + BK) X(t) + B w_x(s(t),t). 
\end{align} 
Consider the Lyapunov function 
\begin{align}\label{eq:V1-def}
    V =&  \frac{3}{4 s_{\rm r}^2} \intst w(x,t)^2 dx  + \frac{1}{2} \intst w_x(x,t)^2 dx \notag\\
    & +  X^\top P X ,  
\end{align}
where $P \in \bbR^{2 \times 2}_+$ is a positive definite matrix satisfying 
\begin{align} \label{eq:Lyapunov}
    P(A + BK)  + (A + BK)^\top P \leq - Q, 
\end{align}
for any given positive semidefinite matrix $Q \in \bbR^{2 \times 2}_{\geq 0}$. Let $Q$ be a diagonal matrix of $Q = \diag(\lambda_1, \lambda_2)$ with $\lambda_1 \geq 0$, $\lambda_2 \geq 0$. 
Taking the time derivative of \eqref{eq:V1-def} along with \eqref{eq:target-special-PDE}--\eqref{eq:target-special-ODE} and applying Young's, Cauchy-Schwarz, and Poincare's inequalities leads to 
\begin{align}
 \dot V  &   \leq - \frac{3 \alpha}{4 s_{\rm r}^2} || w_{x}||^2
    - \frac{ \alpha}{4}  || w_{xx}||^2 - X^\top \Lambda X 
    \notag\\ & \quad + \dot s(t) \left( \frac{9 }{8 s_{\rm r}^3}||w ||^2 + X^\top K K^\top X \right) , \label{eq:Vdot-1}
\end{align}
where 
\begin{align}
    S & = \frac{(c_1 - c_2)^2}{ \beta^2  }
     \left [
    \begin{array}{cc}
       0  & 0 \\
       0  & 1
    \end{array}
    \right] ,  \label{eq:S-def} \\
    \Lambda & = \frac{ \alpha}{64 s_{\rm r} | Q^{-1/2} PB |^2} Q - \frac{9 s_{\rm r}^3}{\alpha}S . \label{eq:Lambda-def}
\end{align}
To show the stability, we need the positive definiteness of $\Lambda $. Let $d_1 := \ep^{-1} c_1$ and $d_2 := \ep^{-1} + c_2$. Then, a solution to the Lyapunov condition \eqref{eq:Lyapunov} can be given by 
\begin{align}
P = \frac{1}{2} \begin{bmatrix} 
\frac{\lambda_1}{ d_2} + \frac{d_1 \lambda_2}{ d_2} + \frac{d_2 \lambda_1}{ d_1} & \frac{\lambda_1}{ d_1} \\[10pt]
\frac{\lambda_1}{ d_1} & \frac{\lambda_1}{ d_1 d_2} + \frac{\lambda_2}{ d_2}
\end{bmatrix}    
\end{align}
Set $\lambda_2 = \kappa_2 \lambda_1$ with $\kappa_2 \geq 0$. 
Since $Q$ is a diagonal positive definite matrix and $S$ defined in \eqref{eq:S-def} has its non-zero value only in $(2,2)$ element, to show the positive definiteness of $\Lambda$ defined in \eqref{eq:Lambda-def}, it suffices to show the positivity of the $(2,2)$ element in $\Lambda$, which leads to the following inequality:  
\begin{align}
 \frac{ d_1^2 d_2^2 }{ \left( d_2^2 
 + 2 d_1 \right) \kappa_2^{-1} + \kappa_2^{-2} + d_1^2  } \geq \frac{144 s_{\rm r}^4(c_1 - c_2)^2}{\alpha^2 \ep^{2}  } \label{eq:Lambda-pos-cond}
\end{align}
Since the left hand side of \eqref{eq:Lambda-pos-cond} is monotonically increasing with respect to $\kappa_2$ with the supremum $d_2^2 = ( \ep^{-1} + c_2)^2$ as $\kappa_2$ goes to infinity, if it holds that 
\begin{align}
     ( \ep^{-1} + c_2)^2 > \frac{144 s_{\rm r}^4(c_1 - c_2)^2}{\alpha^2 \ep^{2}  }, \label{eq:stability-gain-cond}
\end{align}
then, there exists a sufficiently large $\kappa_2>0$ such that the inequality \eqref{eq:Lambda-pos-cond} holds, and hence $\Lambda$ is positive definite. One can show that \eqref{eq:stability-gain-cond} is equivalent to the gain condition \eqref{eq:theorem-cond} stated in Theorem \ref{thm:stability}. 
Once the positive definiteness of $\Lambda$ is shown, consider another Lyapunov function $W(t) = V(t) e^{ -  a s(t)}$ for some positive parameter $a>0$. Applying the same manner as in the stability proof of the Stefan problem in \cite{Shumon19journal}, 
we can show that there exist positive constants $a>0, b>0$ such that $\dot W \leq - b W$ holds,   
which ensures the exponential stability of the closed-loop system and completes the proof of Theorem \ref{thm:stability}. 
\section{Simulation} \label{sec:simulation}

\begin{figure*}[t]
\centering 
\subfloat[The interface position.]
{\includegraphics[width=2.1in]{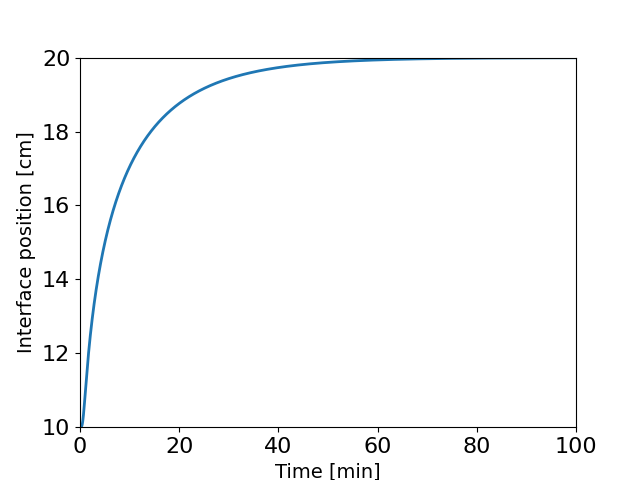}\label{fig:interface}} \hspace{1mm}
\subfloat[The boundary temperature. ]
{\includegraphics[width=2.1in]{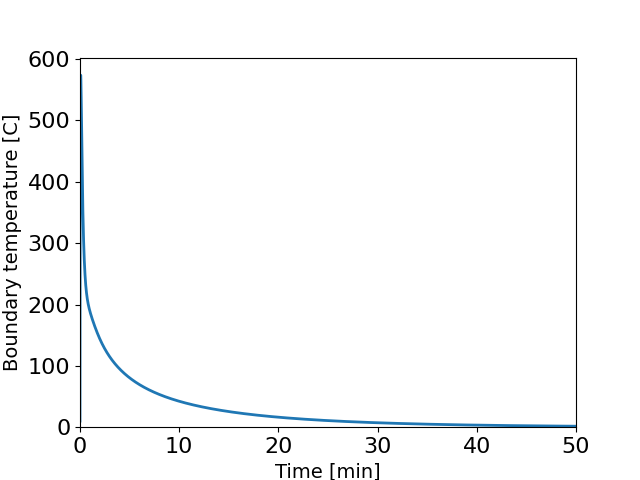}\label{fig:T0}} 
\hspace{1mm}
\subfloat[The control input of boundary heat flux.]
{\includegraphics[width=2.1in]{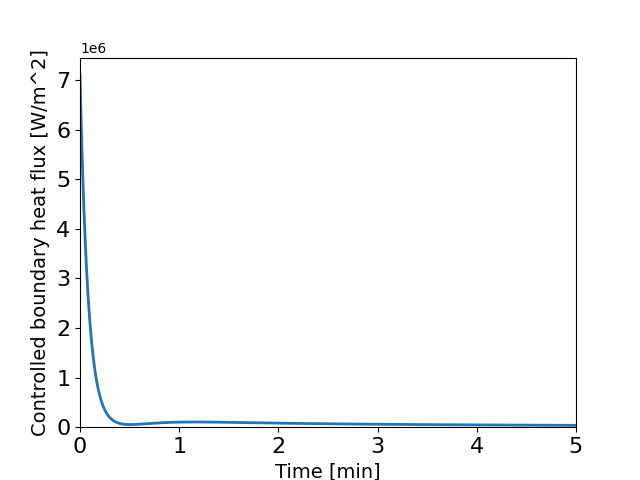}\label{fig:qc}}
\caption{ The closed-loop response of \eqref{eq:stefanPDE}--\eqref{eq:stefanODE} with the control law \eqref{eq:qct-exp}.  }
\label{fig:response}
\end{figure*}

We perform numerical simulation to investigate the effectiveness of the proposed control law \eqref{eq:qct-exp} for the Stefan model with the second-order moving boundary dynamics \eqref{eq:stefanPDE}--\eqref{eq:stefanODE}. Following \cite{Shumon19journal}, we use physical parameters of Zinc. The time constant parameter $\epsilon$ in the model \eqref{eq:stefanODE} is set to $\epsilon = 20$ [s]. The initial conditions and the setpoint are chosen as $s_0$ = 10 [cm], $T_0(x)-T_{{\mathrm m}}= \bar{T}_0(1-x/s_0)$ with $ \bar{T}_0$ = 10 [$^\circ$C], $v_0 = 0$ [m/s], and $s_{{\mathrm r}}$ = 20 [cm]. We can see that the setpoint restriction in Assumption \ref{ass:setpoint} is satisfied. The control gains are set to $c_1 = 0.1$ [1/s] and $c_2 = 0.2$ [1/s].  

The simulation result of the closed-loop response is shown in Fig. \ref{fig:response}. Fig. \ref{fig:interface} illustrates that the interface position converges to the setpoint position monotonically, while the response in the initial time period is relatively slow due to the second-order dynamics \eqref{eq:stefanODE}. Fig. \ref{fig:T0} depicts that the temperature at the boundary blows up first and starts cooling down shortly after the peak is achieved, with maintaining greater value than the melting temperature. Fig. \ref{fig:qc} shows that the boundary heat control maintains positive value, while the response is not monotonic due to the difference between the two gain values of $c_1$ and $c_2$ as seen in \eqref{eq:qcdot}. Therefore, the simulation result is consistent with the theoretical results proven in Proposition \ref{prop:safety} and Theorem \ref{thm:stability}. 
\section{Higher Order Moving Boundary Dynamics} \label{sec:highorder}

In this section, we show an outline for extending the approaches remarked in previous sections to third-order moving boundary dynamics towards generalization. Since a key condition is \eqref{eq:sdot-pos}, we set CBF as 
\begin{align} \label{eq:h1-def}
    h_1(t) = \dot s(t).  
\end{align}
Then, for third-order moving boundary dynamics, the condition of exponential CBF renders the following chain 
\begin{align} \label{eq:h1-ODE}
    \ep_1 \dot h_1(t) &= - h_1 + h_2, \\
 \label{eq:h2-ODE}  \ep_2 \dot h_2(t) &= - h_{2} - \beta T_x(s(t),t), 
\end{align}
where $\ep_1 \geq 0$, $\ep_2\geq 0$
. With this model, clearly, $\ep_1 = \ep_2 = 0$
renders the classical Stefan condition of the first-order dynamics, and either $\ep_1 = 0$ or $\ep_2 = 0$
renders the second-order dynamics shown above. Combining \eqref{eq:h1-def}--\eqref{eq:h2-ODE} gives the moving boundary dynamics
\begin{align}
   \ep_1 \ep_2 \dddot s(t) + ( \ep_1 + \ep_2) \ddot s(t) = - \dot s(t) - \beta T_x(s(t),t). \label{eq:stefan-thrid-ODE}
\end{align}
Therefore, any third-order dynamics having positive coefficients of third- and second-order terms in the left hand side with the same right hand side as \eqref{eq:stefan-thrid-ODE} can be written as a chain of high-order CBF in \eqref{eq:h1-ODE}--\eqref{eq:h2-ODE}. 


The state-space representation of \eqref{eq:stefan-thrid-ODE} by the form of \eqref{eq:u-ODE} is given by 
\begin{align}
    A & = \left [
    \begin{array}{ccc}
       0  & 1 & 0 \\
       0 & 0 & 1 \\
       0  & - \ep_1^{-1} \ep_2^{-1} &  - (\ep_1^{-1} + \ep_2^{-1})
    \end{array}
    \right] , \\
    B^\top &=\left [
    \begin{array}{ccc}
       0  &
       0 &
       - \ep_1^{-1} \ep_2^{-1} \beta
    \end{array}
    \right]  , 
\end{align}
Note that the form of \eqref{eq:qct} of backstepping control and the gain kernel \eqref{eq:phi-sol} are not changed. Thus, we set the control gain vector $K \in \bbR^3$ as 
\begin{align}
    K^\top = \frac{1}{\beta}
    \left [
    \begin{array}{ccc}
       c_1  & 
       (\ep_1 + \ep_2) c_2 & 
       \ep_1 \ep_2 c_3 
    \end{array}
    \right]  .  
\end{align}
Then, the explicit form of the control law is given by 
\begin{align}
    q_{\rm c}(t)
    & = - \frac{ k c_3}{\alpha}  \int_0^{s(t)}  (T(x,t) - T_{\rm m})  dx \notag\\
    & \quad - \frac{k}{\beta} \left (c_1 (s(t) - s_{\rm r}) + c_2 (\ep_1 + \ep_2)  \dot s(t) + c_3 \ep_1 \ep_2 \ddot s(t) \right).  \label{eq:third-qct-exp}
\end{align}
For ensuring safety conditions, we impose the following assumptions. 
\begin{assum} \label{ass:a0}
  It holds that $ a_0 := \ddot s(0) \geq - \frac{v_0}{\ep_1}$. 
\end{assum}
\begin{assum} \label{ass:third-gain12} 
The gains $(c_1, c_2)$ are chosen to satisfy
$0 < c_1 \leq c_2$. 
\end{assum}
\begin{assum} \label{ass:third-setpoint}
The setpoint $s_{\rm r}$ is chosen to satisfy 
    \begin{align}
    s_{\rm r} > \underline{s}_{\rm r}(s_0, v_0, T_0, c_1, c_2)  , \label{eq:third-setpoint}
    \end{align}
    where 
    \begin{eqnarray}
   & \underline{s}_{\rm r}(s_0, v_0, T_0, c_1, c_2) \notag\\
    &:= s_0 + \frac{c_2}{c_1} \left( \ep v_0 + \frac{  \beta}{ \alpha }  \int_0^{s_0}  (T_0(x) - T_{\rm m})  dx \right). \label{eq:third-srunder}
    \end{eqnarray} 
\end{assum}
\begin{assum} \label{ass:third-gain3}
    The gain $c_3$ is chosen to satisfy 
       \begin{align} \label{eq:third-gain3-cond}
        c_2 & \leq c_3 \leq c_2 + \min \left\{\frac{\ep_1}{\ep_2}(c_2 - c_1) , \overline{c}_3, 
        \frac{\ep_2}{\ep_1} c_2
        \right\} , \\
      \overline{c}_3 & :=   \frac{c_1 ( s_{\rm r} - \underline{s}_{\rm r}(s_0, v_0, T_0, c_1, c_2))}{ \ep_1 \ep_2 a_0 +  \frac{ \beta }{\alpha}  \int_0^{s_0}  (T_0(x) - T_{\rm m})  dx }.  
    \end{align}
\end{assum}

Analogous result to Proposition \ref{prop:safety} for ensuring safety condition in third-order model is presented as follows. 
\begin{proposition}
    Let Assumptions \ref{ass:initial}, \ref{ass:CBFinitial}, \ref{ass:a0}--\ref{ass:third-gain3} hold. Then, the closed-loop system of \eqref{eq:stefanPDE}--\eqref{eq:stefanIC} and \eqref{eq:stefan-thrid-ODE} under the control law \eqref{eq:third-qct-exp} satisfies \eqref{temp-valid}, \eqref{eq:qct-valid}, \eqref{eq:sdot-pos}, and \eqref{eq:st-cond}. 
\end{proposition}
\begin{proof}
Assume $q_{\rm c}(t) > 0 $ for all $t \in [0, t^*)$. Then, due to the CBF \eqref{eq:h2-ODE}, we have $h_2(t) = \ep_1 \ddot s(t) + \dot s(t) \geq 0$. Applying this inequality to the time derivative of \eqref{eq:third-qct-exp}  with the help of the left inequality in \eqref{eq:third-gain3-cond} yields
\begin{align}
    \dot q_{\rm c}(t) &\geq - c_3 q_{\rm c}(t) \notag\\
    &+ \frac{k}{\beta} \dot s(t) \left( \left(1 + \frac{\ep_2}{\ep_1} \right) c_2 - c_1  - \frac{ \ep_2}{\ep_1} c_3  \right) . \label{eq:third-qctdot-ineq}
\end{align}
Thus, with the first term in the right inequality in \eqref{eq:third-gain3-cond}, 
\eqref{eq:third-qctdot-ineq} yields $\dot q_{\rm c}(t) \geq - c_3 q_{\rm c}(t)$. Note that two satisfy both inequalities of left-hand-side and the first term in the right-hand-side of \eqref{eq:third-gain3-cond}, we need Assumption \ref{ass:third-gain12}. With Assumption \ref{ass:third-setpoint} and the second term in the right-hand side of \eqref{eq:third-gain3-cond}, it holds that $q_{\rm c}(0) > 0$. Applying this and $\dot q_{\rm c}(t) \geq - c_3 q_{\rm c}(t)$ leads to the contradiction of $q_{\rm c}(t^*) = 0$ similarly to the proof of Proposition \ref{prop:safety}. Thus, \eqref{eq:qct-valid} and \eqref{temp-valid} hold by applying Lemma \ref{lem1}, and \eqref{eq:sdot-pos} holds by exponential CBF chain in \eqref{eq:h1-def}--\eqref{eq:h2-ODE} with Assumption \ref{ass:a0}. Finally, applying \eqref{temp-valid}, \eqref{eq:qct-valid}, \eqref{eq:sdot-pos} and $h_2(t) = \ep \ddot s(t) + \dot s(t) >0$ to the control law \eqref{eq:third-qct-exp}, with the third term in the right-hand side of the inequality \eqref{eq:third-gain3-cond}, one can show \eqref{eq:st-cond} holds. 


\end{proof}

    Since all the steps from \eqref{eq:target-special-PDE} to \eqref{eq:Lambda-def} are not dependent on the order of the moving boundary dynamics, another gain condition is derived by showing the positive definiteness of $\Lambda$ defined in \eqref{eq:Lambda-def}, where $S$ is now given by 
    \begin{align}
        S &= A^\top K K^\top A = \frac{1}{\beta^2}  \left [
    \begin{array}{cc}
       0  & 0 \\
       0  & \Gamma
    \end{array}
    \right] , \\
    \Gamma &= 
     \left [
    \begin{array}{cc}
       \gamma_1^2  & \gamma_1 \gamma_2 \\
      \gamma_1 \gamma_2  & \gamma_2^2
    \end{array}
    \right] , 
    \end{align}
    and  $ \gamma_1 = c_3 - c_1$, and $\gamma_2 = (\ep_1 + \ep_2)(c_3 - c_2)$. 
    Thus, we can see that at least the case $c_1 = c_2 = c_3$ ensures the positive definiteness of $\Lambda$ defined in \eqref{eq:Lambda-def} with maintaining the gain conditions in Assumptions  \ref{ass:third-gain12}--\ref{ass:third-gain3}. Therefore, we can see that there exists a constant $\overline{\overline{c}}_3>0$ such that for any $c_3< c_2 + \overline{\overline{c}}_3$ with satisfying Assumptions  \ref{ass:third-gain12}--\ref{ass:third-gain3} the closed-loop system is globally exponentially stable as proved in Theorem \ref{thm:stability}.  

    
\section{Conclusion} \label{sec:conclusion}

This paper has presented a safe stabilizing boundary control of the Stefan PDE with a high-order moving boundary dynamics. 
We have first considered the second-order moving boundary dynamics and shown that the constraint on the moving boundary is ensured under a positivity condition of a boundary control. 
We have developed a feedback control law to stabilize the moving boundary at a desired position by PDE backstepping with showing that the control constraint is satisfied. 
Then, the global exponential stability of the closed-loop system is shown by Lyapunov analysis by imposing explicit conditions on the control gains related to the initial condition. 
An outline of the extension to the third-order moving boundary dynamics has also been presented. Future work considers QP-safety filter, delay in moving boundary dynamics, and an observer-based output feedback control.


\bibliographystyle{ieeetr}
\bibliography{ref.bib}


\end{document}